\documentclass[12pt]{amsart}

\usepackage{DKstyle}

\newcommand{\vertiii}[1]{{\left\vert\kern-0.25ex\left\vert\kern-0.25ex\left\vert #1
    \right\vert\kern-0.25ex\right\vert\kern-0.25ex\right\vert}}
\theoremstyle{plain}

\usepackage{caption}
\usepackage[labelfont=rm]{subcaption}

\usepackage[pagebackref=true, colorlinks]{hyperref}

\hypersetup{pdffitwindow=true,linkcolor=blue,citecolor=blue,urlcolor=blue,menucolor=blue}

\usepackage{comment}

\title{Shrinking targets for semisimple groups}
\author{Anish Ghosh }
\author{Dubi Kelmer}

\thanks{Anish Ghosh is partially supported by ISF-UGC. Dubi Kelmer is partially supported by NSF grant DMS-1401747.}
\email{ghosh@math.tifr.res.in}
\address{School of Mathematics, Tata Institute of Fundamental Research, Homi Bhabha Road, Colaba, Mumbai 400005, India}
\email{kelmer@bc.edu}
\address{Boston College, Boston, MA}


\date{\today}%
\dedicatory{}%
\commby{}%
 
\begin{document}

\begin{abstract}
We study the shrinking target problem for actions of semisimple groups on homogeneous spaces, with applications to logarithm laws and Diophantine approximation related to an effective version of the Oppenheim conjecture valid for almost all quadratic forms.
\end{abstract}

 \maketitle
 
 \section{Introduction}
Let $H$ be a group acting ergodically on a probability space $(X,\mu)$. In this paper, we are interested in the \emph{shrinking target problem} for the $H$ action.  More explicitly, when considering a growing $H$ ball, we want to know how fast does the ball have to grow, so that a typical orbit will hit a shrinking family of targets. We will show that, in many cases, this typical rate depends only on the dimension of the $H$ balls, and not on the family of shrinking targets. We then give applications to logarithm laws on homogeneous spaces as well as to effective results on the Oppenheim conjecture for generic quadratic forms.

 \subsection{The shrinking target problem}
 Let $\{A_t\}_{t\geq 1}$ denote a decreasing family of measurable subsets of $X$ that are shrinking at a more or less uniform rate $\mu(A_t)\asymp 1/t$.  Here and below we denote $F(t)\asymp G(t)$ if there is some constant $c>1$ such that 
 $$c^{-1}F(t)\leq G(t)\leq cF(t).$$
 We will also denote $F(t)\ll G(t)$ if $F(t)\leq c G(t)$ for some constant $c>0$ and we will use subscripts to indicate the dependance of the constant on parameters.\\
 
\noindent In order to quantify growing $H$ orbits we fix a norm on $H$ and consider growing norm balls $H_t=\{g\in H: \norm{h}\leq t\}$. We assume the following mild regularity on the growth of these norm balls:  There are exponents $d_H^-\leq d_H^+$ such that for any $\epsilon>0$ we have
\begin{equation}\label{e:dH}
t^{d_H^--\epsilon}\ll_\epsilon m(H_t)\ll_\epsilon  t^{d_H^++\epsilon}
\end{equation} where $m$ denotes Haar measure on $H$. We note that in most applications the $H$-balls grow more regularly, with a uniform exponent $d_H=d_H^-=d_H^+$.

\noindent To measure how fast an $H$ orbit has to grow in order to hit the shrinking targets we define for any $x\in X$ the critical exponent 
$$\alpha(x)=\alpha_{H,A}(x),$$ to be the infimum over all $\eta\geq 0$ such that the set $\{h: xh\in A_{\|h\|^\eta}\}$ 
is bounded, if such an exponent exists. 
Under some mild regularity assumptions on the norm (see Lemma \ref{l:inv}) the critical exponent $\alpha(x)$ is invariant under the $H$ action, and hence from ergodicity it is constant almost everywhere. We denote this constant by $\alpha_{H,A}$.  \\

\noindent We say that the shrinking targets are stable under small perturbation in $H$ if $\mu(A_t)\asymp\mu(A_tH_1)$. Assuming this holds, a simple argument using Lebesgue dominated convergence (see Lemma \ref{l:BC}) gives an upper bound
 \begin{equation}\label{e:BC}
 \alpha_{H,A}\leq d_H^+.
 \end{equation}
 \begin{rem}
The assumption that the shrinking sets are stable under small perturbation by $H$ is natural and holds in many cases. Moreover, when it does not hold one can consider thickened shrinking targets given by $\tilde{A}_t=A_tH_1$, then the new shrinking targets are stable, and it is not hard to see that $\alpha_{H,A}=\alpha_{H,\tilde{A}}$.
\end{rem}
\noindent The main result of this paper is to give a lower bound for this exponent. To do this we use an effective mean ergodic theorem (see section \ref{s:MET}  below).
In many cases,  the lower bound we obtain coincides with the upper bound, in which case $\alpha_{H,A}=d_H$, and this holds independently of the family of shrinking targets.

\subsection{Mean Ergodic Theorem} \label{s:MET}
We now recall the notion of a mean ergodic theorem for the $H$ action.
The $H$ action on $X$ leads to a unitary representation of $H$ on $L^{2}(X)$. For $f \in L^{2}(X)$, consider the unitary averaging operator:
 \begin{equation}
 \pi_{t}(f)(x) := \frac{1}{m(H_t)}\int_{H_t}f(xh)dm(h)
 \end{equation}
 \noindent We say that the $H$ action on $X$ satisfies an \emph{effective mean ergodic theorem} if there exists $\kappa > 0$ such that for any $f \in L^{2}(X)$, 
 \begin{equation}\label{eq:met}
 \|\pi_{t}f - \int_{X}fd\mu\|_{2} \ll_\kappa m(H_t)^{-\kappa}\|f\|_2.
 \end{equation}
 
With this notion we can show
\begin{thm}\label{main}
Assuming the $H$ action on $X$ satisfies an effective mean ergodic theorem with exponent $\kappa$, we have a lower bound $\alpha_{H,A}\geq 2\kappa d_H^-$.
\end{thm}
\begin{rem}
The theorem implies that  for any $\eta>2\kappa d_H^-$ for almost all $x\in X$ there are arbitrarily large values of $t$ such that $xH_t\cap A_{t^\eta}\neq \emptyset$.
In fact, this method actually gives a stronger result showing that for almost all $x\in X$, for all $t$ sufficiently large, $xH_T\cap A_{T^\eta}\neq \emptyset$; see Proposition \ref{p:upper}.
\end{rem}
 
\noindent  The mean ergodic theorem above holds in a wide level of generality whenever the representation of $H$ on $L^2(X)$ has a spectral gap, moreover, the constant $\kappa > 0$ above is directly related to the quality of the spectral gap. We refer to \cite{GorodnikNevo10} for an extensive discussion of the general setup and proofs. In particular,   for spaces $X=\G\bk G$ with a locally compact group $G$, lattice subgroups $\Gamma$ of $G$, and unimodular subgroups $H$ of $G$ an effective mean ergodic theorem can be proved. Moreover,  in many cases, \eqref{eq:met}  is satisfied for a triple $(G, H, \G)$ for any $\kappa < 1/2$. We will call such a triple a \emph{tempered triple}, i.e. $(G, H, \G)$ is a tempered triple if (\ref{eq:met}) holds for every $\kappa < 1/2$.  
An extensive list of tempered triples is provided in \cite{GhoshGorodnikNevo14}, for the readers convenience we recall some of their examples below:
\begin{enumerate}
\item Consider  $H = \SL_{2}(\R)$ and $G =  \SL_{3}(\R)$ such that the representation of $\SL_{2}(\R)$ in $\SL_{3}(\R)$ is non-trivial and irreducible. Then the triple $(G, H, \G)$ is tempered for any lattice subgroup $\G$ of $G$. This phenomenon was used by Kazhdan in his original proof of Property T.\\
\item Let $n \geq 3$ and consider $G = \Sp(n,1)$. There is an embedding of $H = \Sp(2,1)$ such that the triple $(G, H, \G)$ is tempered for all lattices $\G$ in $G$.\\
\item Consider $H=\SL_2(\R)$,  $G=\SL_2(\C)$ then $(G,H,\G)$ is a tempered triple for $\G=\SL_2(\Z[i])$. 
\end{enumerate}
\noindent For tempered triples, the upper bound \eqref{e:BC} combined with Theorem \ref{main} gives 
\begin{cor}
 Let $H\leq G$ act on $X=\G\bk G$ and suppose $(G, H, \G)$ is a tempered triple and that $d_H^-=d_H^+=d_H$. Then for any family $\{A_t\}$ of shrinking targets which is stable under small perturbations by $H$, the critical exponent satisfies $\alpha_{A,H}=d_H$. \end{cor}

 \subsection{Logarithm laws}
Let $G$ be a semisimple Lie group and $\G$ be a lattice in $G$. If $H$ is a subgroup of $G$ acting ergodically on $X=\G\bk G$, the critical exponent of an appropriate family of shrinking sets is closely related to logarithm laws. 
Logarithm laws are usually defined for a one parameter flow, measuring the fastest rate of excursions into the cusps, however, they have natural generalization for other group actions of a subgroup $H\leq G$. Moreover, it is also possible to generalize this to study visits of the flow to other shrinking neighborhoods. We describe these two generalizations below:\\

\noindent Assume that $\G \bk G$ is non compact. Given any natural distance function $\dist$ on $\G\bk G$, define a cusp neighborhood 
\begin{equation}\label{e:Bcusp}
B_s(\infty)=\{x\in \G\bk G| \dist(x,o)> s\},
\end{equation}
where $o\in \G\bk G$ is any fixed base point. Similarly, and with no assumptions on the cocompactness of $\G$, for any point $y\in \G\bk G$ we let 
\begin{equation}\label{e:By}
B_s(y)=\{x\in \G\bk G| \dist(x,y)< s\}.
\end{equation}
We then have that  (after appropriately normalizing the distance function) $\mu(B_s(\infty))\asymp e^{-s}$ as $s\to \infty$, and that $\mu(B_s(y))\asymp s^{d_X}$ as $s\to 0$.
To measure how fast the orbit $xH$ makes excursions into the cusp (respectively, approaches the point $y$) we define for any $x\in \G\bk G$
 \begin{equation}
 \beta_t(x)=\sup_{h\in H_t}\dist(xh,o)
 \end{equation}
 and
  \begin{equation}
 \beta_t(x,y)=\inf_{h\in H_t}\dist(xh,y).
 \end{equation}
 
\noindent When the subgroup $H$ is a diagonalizable one parameter group, Kleinbock and Margulis  \cite{KleinbockMargulis1999}  (following the work of Sullivan \cite{Sullivan1982} on the geodesic flow on hyperbolic spaces) used exponential mixing of such flows to  show that (with the above normalization of the distance function) for almost all $x$
 $$\limsup_{t\to\infty}\frac{\beta_t(x)}{\log(t)}=1.$$
 Logarithm laws for diagonalisable one parameter groups acting on the Bruhat Tits building of a semsimple group defined over a local field of positive characteristic were obtained by Athreya, Ghosh and Prasad \cite{AthreyaGhoshPrasad12}.  In the special case when $G$ is of real rank one, $H$ is diagonalizable, and the distance function comes from the hyperbolic distance on the corresponding symmetric space, Maucourant \cite{Maucourant06} considered also shrinking neighborhoods of a point $y$ and showed that for almost all $x\in \G\bk G$ one has
 \begin{equation}\label{e:Maucourant}
  \limsup_{t\to\infty}\frac{-\log\beta_t(x,y)}{\log(t)}=\frac{1}{d_X-1}.
 \end{equation}
 This was later generalized by Hersonsky and Paulin \cite{HersonskyPaulin10} to higher rank situations.\\
 
 For non diagonalizable actions, less is known. Here Athreya and Margulis considered the cases where $H$ is the expanding horospherical group corresponging to a diagonalizable flow \cite{AthreyaMargulis14} as well as the case where $H$ is a one dimensional unipotent group acting on the space of lattices $\SL_n(\Z)\bk \SL_n(\R)$ \cite{AthreyaMargulis09} and in both cases they managed to prove logarithm laws.
 Similar results were also obtained by Kelmer and Mohammadi \cite{KelmerMohammadi12} for cusp excursions of one parameter unipotent flows on spaces of the form $\G\bk G$ where $G$ is a product of a number of copies of $\SL_2(\R)$ and $\SL_2(\C)$ and $\G$ is irreducible, and more recently by Yu \cite{Yu2016} when $G=\SO(n,1)$. For the analogous problems for $\beta_t(x,y)$, to the best of our knowledge, there are no results for non diagonalizable flows.
 
In a somewhat dual setting, one could consider the action of a lattice $\G$ on the homogeneous space $G/H$ where $H$ is a closed subgroup of $G$ and study the shrinking target problem for the $\G$ action on $G/H$. Recently there has been considerable interest in this problem, see \cite{LaurentNogueira12} for example. In particular in \cite{GhoshGorodnikNevo14} Ghosh, Gorodnik and Nevo developed a technique to address this problem using an effective mean ergodic theorem for the $H$ action on $\G\bk G$ and a duality principle. The results in \cite{GhoshGorodnikNevo14} are phrased in terms of a \emph{Diophantine exponent} of a point and in loc. cit. generically best possible values for these exponents are obtained in a wide variety of cases. We refer the reader also to \cite{GhoshGorodnikNevo15} for a survey of these techniques and more examples of estimates for Diophantine exponents for lattice orbits on homogeneous varieties.
  
 We now return to logarithm laws for $H$ actions on $\G\bk G$ and  relate these problems to the more general problem of finding the critical exponent of a family of shrinking targets described above. 
 \begin{prop}\label{prop2}
 For $x \in X:=\G\bk G$, let $\alpha(x)$ denote the critical exponent corresponding to  the shrinking family $A_t=B_{\log t}(\infty)$. Then, for every $x \in X$ such that $\alpha(x)$ is defined, we have
\begin{equation}
\limsup_{t \to \infty} \frac{\beta_{t}(x)}{\log t} = \alpha(x).
\end{equation} 
Similarly, assuming the thickened balls satisfy $\mu(B_{s}(y)H_1)\asymp s^{d_0}$ for some exponent $d_0$, if $\alpha(x,y)$ is the critical exponent corresponding to the shrinking family $A_t=B_{t^{1/d_0}}(y)H_1$, 
then for every $x \in X$ for which
 $\alpha(x,y)$ is defined, either $y\in xH$ or
\begin{equation}
\limsup_{t \to \infty} \frac{-\log\beta_{t}(x,y)}{\log t} = \frac{\alpha(x, y)}{d_{0}}.
\end{equation} 
 \end{prop}
 \begin{rem}
 In many cases the assumption in the second statement holds and $d_0$ can be computed explicitly. In particular,
 when $H$ is a Lie subgroup of $G$ and the distance function is the Riemannian metric on $G$ induced from the Killing form, then $d_0=\dim(G)-\dim(H)$.
 \end{rem}

\noindent Using Proposition \ref{prop2} and Theorem \ref{main} we immediately obtain: 
 \begin{cor}
 Let $H\leq G$ act on $\G\bk G$ and suppose $(G, H, \G)$ is a tempered triple and that $d_H^-=d_H^+=d_H$. Then for almost all $x$,
 \begin{equation}
\limsup_{t \to \infty} \frac{\beta_{t}(x)}{\log t} = d_H
\end{equation} 
and, under the same assumption as in Proposition \ref{prop2}, for all $y\in \G\bk G$ and for almost all $x$,
\begin{equation}
\limsup_{t \to \infty} \frac{-\log\beta_{t}(x,y)}{\log t} = \frac{d_H}{d_0}.
\end{equation} 
 \end{cor}

 \subsection{Diophantine approximation}
 Another application of our result is related to an effective version of the Oppenheim conjecture. Given an irrational indefinite ternary quadratic form, $Q$, the Oppenheim conjecture, proved by Margulis \cite{Margulis1989}, implies that $Q(n)$ with $n\in \Z^3\setminus\{0\}$ takes values arbitrary close to zero. A way to make this quantitative is to ask  how close to zero can $Q(n)$ get when considering only integer vectors with $\norm{n}\leq T$. One can also try and estimate the number of integer vectors of bounded norm with $Q(n)$ small (see e.g. \cite{EskinMargulisMozes1998,EskinMargulisMozes2005,MargulisMohammadi11}), however we do not address this question here.\\

\noindent The problem of making Margulis' result and other related results on the density of values of quadratic forms \emph{effective} is a difficult problem with a long history. One of the main difficulties for establishing effective results is distinguishing between rational forms and irrational forms that are very well approximated by rational ones. We mention the work of Lindenstrauss-Margulis \cite{LindenstraussMargulis14} on this problem, implying that, unless $Q$ is very well approximated by a rational form, the values of $Q(n)$ with  integer vectors $\|n\|\leq T$ can be as small as $O(\frac{1}{\log^\kappa T})$ for some $\kappa>0$ (see also  
 \cite{GotzeMargulis2010} for other effective results on this problem).\\

\noindent While our method does not say anything about any specific form, it does give very strong effective results (i.e., replacing the logarithm by a power) which hold for generic forms (i.e. for almost every form). We should also mention forthcoming results which are closer in spirit, of Ghosh-Gorodnik-Nevo \cite{GhoshGorodnikNevo16} and Athreya-Margulis \cite{AthreyaMargulis16} where very general effective results are proved for generic quadratic forms. \\

\noindent To make the notion of almost every form more precise we need to parametrize the space of indefinite ternary quadratic forms. Recall the action of $\SL_3(\R)$ on forms is given by 
$$Q^g(v)=Q(vg),$$
with $g\in \SL_3(\R)$ acting on $v\in \R^3$ linearly. We say that two forms are equivalent if $Q_1=\lambda Q_2^\g$ with $\lambda\in \R$ and  $\g\in\SL_3(\Z)$. Note that equivalent forms take the same values on $\Z^3$ after scaling by a constant. To avoid the scaling ambiguity we restrict to forms of determinant one. In order to parametrize the space of determinant one forms up to equivalence fix a form $Q_0(v)$ given by
\begin{equation} \label{e:Q0}Q_0(x,y,z)=x^2+y^2-z^2,
\end{equation}
and note that any determinant one indefinite ternary quadratic form is given by $Q=Q_0^g$ with some $g\in \SL_3(\R)$. Moreover, two such forms  $Q_0^{g},\;Q_0^{g'}$ are equivalent if and only if $g'=\g g$ with $\g\in \SL_3(\Z)$.  We can thus parametrize the space of determinant one indefinite ternary quadratic forms up to equivalence by the space of lattices in $\R^3$
$$X_3=\SL_3(\Z)\bk \SL_3(\R)$$
 where a point $x=\G g\in X_3$ corresponds to the lattice $\Lambda=\Z^3g$ and to the form $Q=Q_0^g$. Consequently, the probability measure $\mu$ on $X_3$ coming from Haar measure on $\SL_3(\R)$ gives us a natural measure on the space of forms.\\

\noindent With this parametrization we have 
\begin{thm}\label{t:Opp}
Fix a norm on $\R^3$ and a  positive  $\tau<1$. Then,  for $\mu$-almost all $g\in X_3$ for all sufficiently large $T$ there is $n\in \Z^3$ with $\norm{n}<T$ and 
\begin{equation}
|Q_0^g(n)|\ll_{g} T^{-\tau}
\end{equation}
\end{thm}
\begin{rem}
The threshold of $\tau\leq1$ for the exponent is natural in the sense that there are roughly $T^3$ integer vectors with  $\norm{n}\leq T$ and $Q(n)$ takes values of order $\ll T^2$, so the average spacing between values of $Q$ is of order $T^{-1}$.  \end{rem}

\begin{rem}
Recently, Bourgain \cite{Bourgain16} proved a similar result for almost all diagonal quadratic forms. That is, he showed that  for any $\tau<2/5$, (respectively for any $\tau<1$ assuming the Lindel\"of hypothesis for the Riemann zeta function)  and for almost all $\beta$,  the form  $Q_\beta(x,y,z)=x^2+\alpha y^2-\beta z^2$ satisfies that for all sufficiently large $T$ there is $n\in \Z^2$ with $\norm{n}<T$ and $|Q_\beta(n)|\ll T^{-\tau}$.
\end{rem}

%
 \section{Proofs}
\noindent In this section we collect all the proofs for the statements made above.

\subsection{Shrinking targets}
First, to show that the critical exponent is indeed $H$ invariant we show 
\begin{lem}\label{l:inv}
Assume that the norm on $H$ satisfies the following regularity condition: For any $h_0,h\in H$ we have $\|h_0h\|\asymp_{h_0} \|h\|$. Then $\alpha(x)=\alpha(xh)$ for any $h\in H$.
\end{lem}
\begin{proof}
Let $x\in X$ such that $\alpha(x)$ is defined and let $h_0\in H$. Note that if $\eta_1>0$ satisfies that 
$$\{h: xh\in A_{\|h\|^{\eta_1}}\}$$
is bounded, and $\eta_2>\eta_1$ then
$$\{h: xh_0h\in A_{\|h\|^{\eta_2}}\}$$
is also bounded. Indeed, if not then
$$\{h: xh\in A_{\norm{h_0^{-1}h}^{\eta_2}}\}$$
is unbounded and since $\norm{h_0^{-1}h}\geq c_0\norm{h}$, when $\|h\|$ is sufficiently large
$$\|h_0^{-1}h\|^{\eta_2}\geq \|h\|^{\eta_1},$$
implying that
$$\{h: xh\in A_{\|h\|^{\eta_1}}\},$$
is unbounded, in contradiction to our assumption. 
 
From this we see that $\alpha(xh_0)$ is also defined, and that $\alpha(x)\leq \alpha(xh_0)$. Since the argument is symmetric we conclude that $\alpha(x)=\alpha (xh_0)$.
\end{proof}

Next, to establish the lower bound \eqref{e:BC} we show the following (cf. \cite{Maucourant06})

\begin{lem}\label{l:BC}
Let $A_t$ be a decreasing family of measurable subsets of $X$ satisfying $\mu(A_t H_\delta)\asymp  \frac{1}{t}$ for some fixed $\delta>0$. 
 Then for all $\eta>d_H^+$ for almost every $x \in X$,
 \begin{equation}
 \{h \in H~:~xh \in A_{\|h\|^\eta}\}
 \end{equation}
 \noindent is bounded.
 \end{lem}
 \begin{proof}
 Let $B_t=A_t H_\delta$ and let $f_t$ denote the indicator function of $B_t$.
 For any $T>0$ and $x\in X$ define the function 
 $$F_T(x)=\int_{H_T}f_{\|h\|^\eta}(xh)dm(h),$$
  and let $F_\infty(x)=\lim_{T\to\infty}F_T(x)\in[0,\infty]$.\\
  
\noindent We show that if the set $\{h \in H~:~xh \in A_{\|h\|^\eta}\}$ is unbounded then $F_\infty(x)=\infty$. Indeed, let $h_k\in H$ be an unbounded sequence with $xh_k\in A_{\|h_k\|^{\eta}}$. Then for any $h\in h_kH_\delta$ we have that 
  $xh\in B_{\|h\|^\eta}$. Perhaps after taking a subsequence we may assume the sets $h_kH_\delta$ are all disjoint and hence 
  the set $\{h\in H: xh\in B_{\|h\|^\eta}\}$ has infinite measure, so $F_\infty(x)=\infty$ as claimed.\\

\noindent  Next, since $\eta>d_H^+$ we have that  
  \begin{eqnarray*}
  \int_{X} F_T(x)d\mu(x)&=&\int_X\int_{H_T}f_{\|h\|^\eta}(xh)dm_H(h)d\mu(x)\\
  &=&\int_{H_T}\mu(B_{\|h\|^\eta})dm(h)\\
  &\leq & \int_{H} \mu(B_{\|h\|^{\eta}})dm(h)\\
  &\ll&  \int_{H} \|h\|^{-\eta}dm(h)  <\infty
  \end{eqnarray*}
  so by dominated convergence for almost all $x\in X$  we have $F_\infty(x)<\infty$, concluding the proof.  
 \end{proof}

%

Finally, the proof of Theorem \ref{main} follows immediately from the following
\begin{prop}\label{p:upper}
 Assume that  the $H$ action on $X$ satisfies an effective mean ergodic theorem with exponent $\kappa$.
 Let $A_t$ be a decreasing family of measurable subsets of $X$ as above.  Then for all $\eta < 2\kappa d_H^-$
\noindent and for almost every $x \in X$ for all sufficiently large $T$ there is $h\in H_T$ with $xh\in A_{T^\eta}$.
 \end{prop}
 
 \begin{proof}
 Let $\eta <\alpha< 2\kappa d_H^-$ and note that for all sufficiently large $T$, if $k=[T]$ then $xH_T\cap A_{T^\eta}$ contains $xH_k\cap A_{k^\alpha}$.
 It is thus sufficient to show that for almost every $x \in X$ for all sufficiently large $k\in \N$ there is $h\in H_k$ with $xh\in A_{k^\alpha}$.

For any set $B\subseteq X$ let  $\cC_{T,B}=\{x\in X: xH_T\cap B=\emptyset\}$, then the effective mean ergodic theorem implies that 
$$\mu(\cC_{T,B})\leq \frac{1}{\mu(B)m(H_T)^{2\kappa}}.$$
Indeed,  if $f$ is the indicator function of $B$ then $\beta_T f(x)=0$ for all $x\in \cC_{T,B}$ and hence
$$\mu(B)^2\mu(\cC_{T,B})\leq \| \beta_T(f)-\mu(B)\|_2^2\leq \frac{\mu(B)}{m(H_T)^{2\kappa}}.$$

Let $\cC$ be the set of all points such that for all $T>0$ there is an integer $k\geq T$ such that $H_{k}\cap A_{k^\alpha}=\emptyset$.
That is $\cC=\bigcap_{T>0}\bigcup_{k>T}\cC_{k,A_{k^\alpha}}$. Now note that 
$$\bigcup_{k=T}^{2T}\cC_{k,A_{k^\alpha}}=\{x\in X: \exists ~k~\in [T,2T],\; xH_k\cap A_{k^\alpha}=\emptyset\}\subseteq \cC_{T,A_{(2T)^\alpha}},$$
so that 
$$\cC\subseteq \bigcap_{T>0}\bigcup_{k>\log(T)} \cC_{2^k,A_{2^{k\alpha+1}}}.$$
By the first part, we have that 
$$\mu(\cC_{2^k ,A_{2^{k\alpha+1}}})\leq \frac{2^{k\alpha}}{m(H_{2^k})^{2\kappa}}\ll_\epsilon 2^{k(\alpha-2\kappa d_H^-+\epsilon)},$$
and taking $\epsilon=\frac{2\kappa d_H^--\alpha}{2}>0$ implies $\mu(\cC_{2^k,A_{2^{k\alpha+1}}})\ll  2^{-k\epsilon}$ and hence 
$$\mu(C) \leq \mu(\bigcup_{k>\log(T)} \cC_{2^k,A_{2^{k\alpha+1}}})\ll T^{-\epsilon},$$ 
for all $T$, completing the proof.
 \end{proof}

 \begin{rem}
 Note that we only use Lemma \ref{l:inv} to show that the critical exponent is constant almost everywhere, but we do not use it in the proofs of Lemma \ref{l:BC} and Proposition \ref{p:upper}. In particular, when $d_H^-=d_H^+=d_H$ and we have a tempered triple we get that $\alpha_{H,A}=d_H$ is constant almost everywhere even without the regularity assumption used in the proof of Lemma \ref{l:inv}.
  \end{rem}
  
 \subsection{Logarithm laws}
We now establish the connection between the critical exponent and logarithm laws by giving the 
 \begin{proof}[Proof of Proposition \ref{prop2}]
  For the first statement, suppose $\limsup_{t \to \infty} \frac{\beta_{t}(x)}{\log t} < \alpha(x)$. Then, there exist $\eta < \alpha(x)$ such that  for all sufficiently large $t\geq t_0$,
 $$ \beta_{t}(x) < \eta \log(t).$$
Therefore for $t>t_0$ for every $h \in H_t$, $\dist(xh, o) < \eta\log t$. In particular for $\|h\|=t\geq t_0$ we have $xh \notin B_{\eta\log(t)}(\infty)=A_{\|h\|^\eta}$, hence, $\{h~:~ xh \in A_{\|h\|^{\eta}}\}$ is a bounded set. So $\alpha(x) < \eta$ which is a contradiction. The other direction follows by a similar argument.\\

\noindent For the second statement, let $x,y\in X$ with $\alpha(x,y)$ defined and assume they are not in the same $H$ orbit (if $y\in Hx$ then $\beta_t(x,y)=0$ for all sufficiently large $t$ so there is nothing to show).  Assume first that  $\limsup_{t \to \infty} \frac{-\log\beta_{t}(x,y)}{\log t} > \frac{\alpha(x,y)}{d_X}$. Then there exist $\eta>\frac{\alpha(x,y)}{d_0}$ and a sequence $t_k\to\infty$ with 
 $\beta_{t_k}(x,y)<t_k^{-\eta}$. This implies that there is a sequence $h_k\in H_{t_k}$ with 
 $\dist(xh_k,y)<t_k^{-\eta}$ hence, for $k$ sufficiently large
 $$xh_k\in B_{t_k^{-\eta}}(y)= A_{t_k^{\eta d_0}}.$$  
 Since $t_k\to\infty$ we get that $xh_k\to y$ and since we assume that $y\not\in Hx$ we must have that $\|h_k\|\to\infty$ as well.  Since $A_{t_k^{\eta d_0}}\subseteq A_{\|h_k\|^{\eta d_0}}$ we get that  $xh_k\in A_{\|h_k\|^{\eta d_0}}$ and hence the set
 $$\{h\in H: xh\in A_{\|h_k\|^{\eta d_0}}\},$$
 is unbounded in contradiction. Again, the other direction is similar.
 \end{proof}
\subsection{Diophantine approximation}
We conclude the paper by showing how Theorem \ref{t:Opp} follows from a special case of Proposition \ref{p:upper}.\\

\noindent Let $H=\SL_2(\R),\; G=\SL_3(\R)$ and $\G=\SL_3(\Z)$. Let $Q_0$ be as in \eqref{e:Q0} and consider the  double spin cover map $\iota:H\to \SO_{Q_0}$ given by
\begin{equation}\label{e:spin}
\iota\begin{pmatrix} a& b\\ c & d\end{pmatrix}=
\begin{pmatrix}\frac{a^2-b^2-c^2+d^2}{2} & ac-bd & \frac{a^2-b^2+c^2-d^2}{2}\\
ab-cs &bc+ad& ab+cd\\
\frac{a^2+b^2-c^2-d^2}{2}& ac+bd& \frac{a^2+b^2+c^2+d^2}{2}
\end{pmatrix}
\end{equation}
This gives an irreducible right action of $H$ on $G$ and hence also on $X_3=\G\bk G$.\\

\noindent We give a norm on $H$ using the spin cover map by defining
\begin{equation}\label{e:norm}
\norm{h}:=\|\iota (h)^{-1}\|_2
\end{equation}
where $\|\cdot\|_2$ denotes the the Hilbert-Schmidt norm on $G$ given by 
$\norm{g}_2^2=\tr(g^tg)$. 
\begin{lem}
With this norm we have that 
$m(H_t)\asymp t,$
so that $d_H^+=d_H^-=1$.
\end{lem} 
\begin{proof}
Consider the $KA^+K$ decomposition of $H$ with $K=\SO(2)$ and 
$$A^+=\left\{a_t=\begin{pmatrix} e^{t/2} & 0\\ 0 & e^{-t/2}\end{pmatrix}: t\geq 0\right\}.$$
Let $k_\theta=\left(\begin{smallmatrix} \cos(\theta) & \sin(\theta)\\ -\sin(\theta)& \cos(\theta)\end{smallmatrix}\right)$ parametrize $K$.
In the coordinates $h=k_\theta a_tk_{\theta'}$  the Haar measure of $H$ is given by
$$dm(h)=\sinh(t)d\theta d\theta'dt,$$
Moreover, a direct computation using \eqref{e:spin} shows that for $h=k_\theta a_tk_{\theta'}$ we have
$$\norm{h}^2=1+2(1+2\sinh^2(t))(1-\sin^2(\theta)\cos^2(\theta)).$$
In particular, for large $t\gg1$ we have that $\|h\|\asymp e^t$ and hence $m(H_t)\asymp t$ as claimed.
\end{proof}

After setting up the norm balls,  $H_t$, we need to define our shrinking targets. To do this, given a lattice $\Lambda\in \R^3$ let 
$$\alpha_1(\Lambda)=\sup_{v\in\Lambda\setminus \{0\}}\frac{1}{\norm{v}}.$$
We recall (see e.g. \cite[Section 7]{KleinbockMargulis1999} that 
$$\mu\{\Lambda\in X_3: \alpha_1(\Lambda)\geq t\}\asymp t^3,$$
and we can thus take our shrinking targets to be 
\begin{equation}\label{e:At}
A_t=\{\Lambda\in X_3: \alpha_1(\Lambda)\geq t^{1/3}\}.
\end{equation}
With these preliminaries in place we can give the 

\begin{proof}[Proof of Theorem \ref{t:Opp}.]
Let $H, G$ and $\G$ be as above, and note that the action of $H$ on $\G\bk G$ satisfies a quantitive mean ergodic theorem with any exponent $\kappa<1/2$.  Hence, by Proposition \ref{p:upper}, for any $\eta<1$, for almost all $x\in X_3$ for all $T\geq T_0(x)$ there is $h\in H$ with $\|h\|\leq T$ and  $x\iota(h)\in A_{T^\eta}$. In particular this holds for $\eta=\frac{3\tau}{\tau+2}$ for any positive $\tau<1$.\\

\noindent Let $x=\G g$ be such a point, and let $Q=Q_0^g$ denote the corresponding quadratic form. For $T\geq T_0(x)$,  let $h\in H$ with $\|h\|\leq T$ and  $x\iota(h)\in A_{T^\eta}$. Let $\Lambda=\Z^3g\iota(h)$ and let $v\in \Lambda$ be a vector of shortest length. Then, since $x\iota(h)\in A_{T^\eta}$, we have that $\norm{v}_2=\frac{1}{\alpha_1(\Lambda)}\leq \frac{1}{\norm{h}^{\eta/3}}$. 
We can write $v=ng\iota(h)$ for some $n\in \Z^3$ then, on one hand, 
$$\norm{n}=\norm{v\iota(h^{-1})g^{-1}}\leq \norm{v}\norm{\iota(h^{-1})}_2\norm{g}_2\leq \|g\|_2 T^{\frac{3-\eta}{3}}.$$
On the other hand we have 
$|Q(n)|=|Q_0(v)|\leq \norm{v}^2\leq T^{-2\eta/3} $. So, setting $\tilde{T}=\|g\|_2 T^{\frac{3-\eta}{3}}$ we get that for all $\tilde{T}$ sufficiently large there is $n\in \Z^3$ with $\norm{n}\leq \tilde{T}$ and $|Q(n)|\ll_g \tilde{T}^{-\tau}$.

%
\end{proof}

 \textbf{Acknowledgements} AG thanks Jayadev Athreya for helpful discussions.\\

\end{document}